\documentclass[a4paper,12pt,oneside]{amsart}

\usepackage{amssymb}
\usepackage{bbm}
\usepackage{wasysym}
\usepackage[width=16cm]{geometry}

\normalsize
\newlength{\aufzleft}
\newenvironment{aufz}{\begin{list}{}{\setlength{\listparindent}{0pt}\setlength{\itemsep}{\topsep}\setlength{\labelwidth}{3.2ex}\setlength{\aufzleft}{\labelsep}\addtolength{\aufzleft}{\labelwidth}\setlength{\leftmargin}{\aufzleft}}}{\end{list}}	
\newenvironment{equi}{\begin{list}{}{\setlength{\listparindent}{0pt}\setlength{\itemsep}{\topsep}\setlength{\labelwidth}{4.1ex}\setlength{\aufzleft}{\labelsep}\addtolength{\aufzleft}{\labelwidth}\setlength{\leftmargin}{\aufzleft}}}{\end{list}}	

\newtheoremstyle{bracket}{1ex}{2ex}{\rm}{}{\bfseries}{}{0.8em}{\thmnumber{(#2)}}
\newtheoremstyle{example}{1ex}{2ex}{\rm}{}{\bfseries}{}{0.8em}{\thmnumber{(#2)}\thmname{ #1}}
\newtheoremstyle{thm}{1ex}{2ex}{\itshape}{}{\bfseries}{}{0.9em}{\thmnumber{(#2)}\thmname{ #1}\thmnote{ (#3)}}

\theoremstyle{bracket}
\newtheorem{no}{}[section]

\theoremstyle{example}

\newtheorem{exas}[no]{Examples}
\newtheorem{qu}[no]{Question}
\newtheorem{qus}[no]{Questions}

\theoremstyle{thm}
\newtheorem{lemma}[no]{Lemma}
\newtheorem{prop}[no]{Proposition}

\newtheorem{cor}[no]{Corollary}

\DeclareMathOperator{\ass}{Ass}
\DeclareMathOperator{\var}{Var}
\DeclareMathOperator{\spec}{Spec}
\DeclareMathOperator{\nil}{Nil}

\newcommand{\hm}[3]{{\rm Hom}_{#1}(#2,#3)}
\newcommand{\dfgl}{\mathrel{\mathop:}=}
\newcommand{\Id}{{\rm Id}}
\newcommand{\res}{\!\upharpoonright}
\newcommand{\N}{\mathbbm{N}}
\newcommand{\Z}{\mathbbm{Z}}
\newcommand{\Q}{\mathbbm{Q}}
\newcommand{\assf}{\ass^{\rm f}}
\newcommand{\ia}{\mathfrak{a}}
\newcommand{\ib}{\mathfrak{b}}
\newcommand{\ic}{\mathfrak{c}}
\newcommand{\id}{\mathfrak{d}}
\newcommand{\ip}{\mathfrak{p}}
\newcommand{\iq}{\mathfrak{q}}
\newcommand{\im}{\mathfrak{m}}
\newcommand{\inn}{\mathfrak{n}}
\newcommand{\catmod}{{\sf Mod}}

\newcommand{\ilim}{\varinjlim}
\newcommand{\ext}[4]{{\rm Ext}^{#1}_{#2}(#3,#4)}
\newcommand{\sq}{\hskip1pt\raisebox{.225ex}{\rule{.8ex}{.8ex}\hskip1pt}}

\newcommand{\snf}{\renewcommand{\thefootnote}{*}\footnotetext{PHQ was partially supported by NAFOSTED (Vietnam), grant no. 101.04-2014.25. An earlier version of this article was cited as {\it Bad behaviour of injective modules.}}}

\begin{document}

\title{Injective modules and torsion functors\protect\snf}
\author{Pham Hung Quy}
\author{Fred Rohrer}
\address{Department of Mathematics, FPT University, 8 Ton That Thuyet, 10307 Hanoi, Vietnam}
\email{quyph@fpt.edu.vn}
\address{Grosse Grof 9, 9470 Buchs, Switzerland}
\email{fredrohrer@math.ch}
\subjclass[2010]{Primary 13C11; Secondary 13D45}
\keywords{ITI, non-noetherian ring, injective module, torsion functor, local cohomology, weakly proregular ideal}

\begin{abstract}
A commutative ring is said to have ITI with respect to an ideal $\ia$ if the $\ia$-torsion functor preserves injectivity of modules. Classes of rings with ITI or without ITI with respect to certain sets of ideals are identified. Behaviour of ITI under formation of rings of fractions, tensor products and idealisation is studied. Applications to local cohomology over non-noetherian rings are given.
\end{abstract}

\maketitle


\section*{Introduction}

Let $R$ be a ring\footnote{Throughout the following rings are understood to be commutative.}. It is an interesting phenomenon that the behaviour of injective $R$-modules is related to noetherianness of $R$. For example, by results of Bass, Matlis and Papp (\cite[3.46; 3.48]{lam2}) the following statements are both equivalent to $R$ being noetherian:\linebreak (i) Direct sums of injective $R$-modules are injective; (ii) Every injective $R$-module is a direct sum of indecomposable injective $R$-modules.

In this article, we investigate a further property of the class of injective $R$-modules, dependent on an ideal $\ia\subseteq R$, that is shared by all noetherian rings without characterising them: We say that $R$ has {\it ITI with respect to $\ia$} if the $\ia$-torsion submodule of any injective $R$-module is again injective\footnote{ITI stands for ``injective torsion of injectives''.}. It is well-known that ITI with respect to $\ia$ implies that every $\ia$-torsion $R$-module has an injective resolution whose components are $\ia$-torsion modules. We show below (\ref{1.10}) that these two properties are in fact equivalent.

Our interest in ITI properties stems from the study of the theory of local cohomology (i.e., the right derived cohomological functor of the $\ia$-torsion functor). Usually, local cohomology is developed over a noetherian ring (e.g., \cite{bs}), although its creators were obviously interested in a more general theory -- see \cite{sga2}. While some deeper theorems about local cohomology might indeed rely on noetherianness, setting up its basic machinery does not need such a hypothesis at all. The ITI property with respect to the supporting ideal serves as a convenient substitute for noetherianness, and thus identifying non-noetherian rings with ITI is a first step in extending applicability of local cohomology beyond noetherian rings, in accordance with natural demands from modern algebraic geometry and homological algebra.\footnote{Concerning ``theorems that might indeed rely on noetherianness'', let us mention that in \ref{3.80} we exhibit rings with ITI with respect to every ideal but without Grothendieck vanishing.} Besides the approach via ITI, there is also the extension of local cohomology to non-noetherian rings via the notion of weak proregularity. This originates in \cite{sga2}, was studied by several authors (\cite{lipman}, \cite{yekutieli}, \cite{schenzel}), and is quite successful from the point of view of applications. However -- and in contrast to ITI -- it applies only to supporting ideals of finite type, and when working in a non-noetherian setting, it seems artificial to restrict ones attention a priori to ideals of finite type.

\medskip

The first section contains mostly positive results. Besides giving a new proof of the fact that noetherian rings have ITI with respect to every ideal (\ref{1.150}), we prove that absolutely flat rings have ITI with respect to every ideal (\ref{1.120}) and that $1$-dimensional local domains (e.g. valuation rings of height $1$) have ITI with respect to ideals of finite type (\ref{1.180}). We also show that ITI properties are preserved by formation of rings of fractions (\ref{1.80}) and discuss further necessary or sufficient conditions for ITI.

Unfortunately, there are a lot of nice rings without ITI, even with respect to some principal ideals. In the second section, we provide several such examples (\ref{2.20}, \ref{2.30}, \ref{2.90}) and also show that ITI may be lost by natural constructions such as idealisation (\ref{2.100}) or base change (\ref{2.50}), even when performed on noetherian rings.

The goal of the third section is twofold. First, we show that ITI with respect to an ideal $\ia$ of finite type is strictly stronger than weak proregularity of $\ia$ (\ref{wpr10}, \ref{wpr20}). Second, to round off, we sketch how to swap ITI for noetherianness in some basic results on local cohomology. 

\medskip

Our study of rings with or without ITI left us with at least as much questions as we found answers; some of them are pointed out in the following. We consider this work as a starting point in the search for non-noetherian rings with ITI and hope this naturally arising and interesting class of rings will be further studied and better understood.

\medskip

General notation and terminology follows Bourbaki's {\it \'El\'ements de math\'ematique;} concerning local cohomology we follow Brodmann and Sharp (\cite{bs}).


\section{Rings with ITI}

Let $R$ be a ring and let $\ia\subseteq R$ be an ideal. We choose for every $R$-module $M$ an injective hull $e_R(M)\colon M\rightarrowtail E_R(M)$ with $e_R(M)=\Id_M$ if $M$ is injective; for basics on injective hulls we refer the reader to \cite[X.1.9]{a}. The $\ia$-torsion functor $\Gamma_{\ia}$ is defined as the subfunctor of the identity functor on the category of $R$-modules with $\Gamma_{\ia}(M)=\bigcup_{n\in\N}(0:_M\ia^n)$ for each $R$-module $M$. For basics on torsion functors we refer the reader to \cite{bs}, but not without a word of warning that over non-noetherian rings these functors may behave differently than over noetherian ones -- see \cite{r-torsion}.\footnote{Although some of the following can be expressed in the language of torsion theories we avoid this mainly because $\Gamma_{\ia}$ is not necessarily a radical if $\ia$ is not of finite type -- see \ref{2.120}.}

\begin{prop}\label{1.10}
The following statements are equivalent:
\begin{equi}
\item[(i)] $\Gamma_{\ia}(M)$ is injective for any injective $R$-module $M$;
\item[(ii)] Every $\ia$-torsion $R$-module has an injective resolution whose components are $\ia$-torsion modules;
\item[(iii)] $E_R(\Gamma_{\ia}(E_R(M)))=\Gamma_{\ia}(E_R(M))$ for every $R$-module $M$;
\item[(iv)] $\Gamma_{\ia}(E_R(\Gamma_{\ia}(M)))=E_R(\Gamma_{\ia}(M))$ for every $R$-module $M$.
\end{equi}
\end{prop}

\begin{proof}
The equivalences ``(i)$\Leftrightarrow$(iii)'' and ``(ii)$\Leftrightarrow$(iv)'' are immediate.

``(iii)$\Rightarrow$(iv)'': Let $M$ be an $R$-module. Then, $$\Gamma_{\ia}(M)\subseteq\Gamma_{\ia}(E_R(\Gamma_{\ia}(M)))\subseteq E_R(\Gamma_{\ia}(M)).$$ Taking injective hulls yields $$E_R(\Gamma_{\ia}(M))\subseteq E_R(\Gamma_{\ia}(E_R(\Gamma_{\ia}(M))))\subseteq E_R(\Gamma_{\ia}(M)).$$ As the $R$-module in the middle equals $\Gamma_{\ia}(E_R(\Gamma_{\ia}(M)))$ by (iii) we get the desired equality.

``(iv)$\Rightarrow$(iii)'': Let $M$ be an $R$-module. Then, $$\Gamma_{\ia}(E_R(M))\subseteq E_R(\Gamma_{\ia}(E_R(M)))\subseteq E_R(M).$$ Taking $\ia$-torsion submodules yields $$\Gamma_{\ia}(E_R(M))\subseteq\Gamma_{\ia}(E_R(\Gamma_{\ia}(E_R(M))))\subseteq\Gamma_{\ia}(E_R(M)).$$ As the $R$-module in the middle equals $E_R(\Gamma_{\ia}(E_R(M)))$ by (iv) we get the desired equality.
\end{proof}

We say that $R$ has {\it ITI with respect to $\ia$} if the statements (i)--(iv) in \ref{1.10} hold. It is well-known that noetherian rings have ITI with respect to every ideal (\cite[2.1.4]{bs}); we give a new proof of this fact below (\ref{1.150}). But let us begin by exhibiting first examples -- albeit silly ones -- of possibly non-noetherian rings with ITI properties.

\begin{exas}\label{1.20}
A) Every ring has ITI with respect to nilpotent ideals.

\smallskip

B) If $R$ is a $0$-dimensional local ring, then proper ideals of finite type are nilpotent, and hence $R$ has ITI with respect to ideals of finite type.

\smallskip

C) If $K$ is a field, then $K[(X_i)_{i\in\N}]/\langle X_iX_j\mid i,j\in\N\rangle_{K[(X_i)_{i\in\N}]}$ is a non-noetherian $0$-dimensional local ring whose proper ideals are all nilpotent, hence it has ITI with respect to every ideal.\hfill$\bullet$
\end{exas}

Next we look at a necessary condition for $R$ to have ITI with respect to $\ia$ which proves to be very useful in producing rings without ITI later on. Moreover, the question about its sufficiency is related to the question whether ITI properties are preserved along surjective epimorphisms.

\begin{prop}\label{1.30}
If $R$ has ITI with respect to $\ia$, then $E_R(R/\ia)$ is an $\ia$-torsion module.
\end{prop}

\begin{proof}
Immediately from \ref{1.10} (iv), since $R/\ia$ is an $\ia$-torsion module.
\end{proof}

\begin{lemma}\label{1.40}
If $\ia,\ib\subseteq R$ are ideals with $\ib\subseteq\ia$ and $M$ is an $R$-module with $\ib M=0$ such that $E_R(M)$ is an $\ia$-torsion module, then $E_{R/\ib}(M)$ is an $\ia/\ib$-torsion module.
\end{lemma}

\begin{proof}
Straightforward on use of the canonical isomorphism of $R/\ib$-modules\linebreak $\hm{R}{R/\ib}{E_R(M)}\cong E_{R/\ib}(M)$ (\cite[10.1.16]{bs}).
\end{proof}

\begin{qu}\label{1.50}
It would be useful to know whether the converse of \ref{1.30} holds, i.e.:
\begin{aufz}
\item[(A)] {\it Suppose $E_R(R/\ia)$ is an $\ia$-torsion module. Does then $R$ have ITI with respect to $\ia$?}
\end{aufz}
For example, if $\ib\subseteq R$ is an ideal with $\ib\subseteq\ia$ and $E_R(R/\ia)$ is an $\ia$-torsion module, then $E_{R/\ib}((R/\ib)/(\ia/\ib))$ is an $\ia/\ib$-torsion module by \ref{1.40}. So, if Question (A) could be positively answered and $R$ has ITI with respect to $\ia$, then $R/\ib$ would have ITI with respect to $\ia/\ib$. In particular, ITI properties would be preserved along surjective epimorphisms.\hfill$\bullet$
\end{qu}

While we do not know whether or not ITI properties are preserved along surjective epimorphisms (let alone arbitrary epimorphisms), we show now that they are preserved along {\it flat} epimorphisms, hence in particular by formation of rings of fractions.\footnote{This result could raise hope to obtain non-noetherian rings with ITI by considering flat epimorphisms with noetherian source. Alas, Lazard showed that flat epimorphisms with noetherian source have noetherian target (\cite[IV.2.3]{lazard}).}

\begin{prop}\label{1.60}
Let $h\colon R\rightarrow S$ be a flat epimorphism of rings. If $R$ has ITI with respect to $\ia$, then $S$ has ITI with respect to $\ia S$.
\end{prop}

\begin{proof}
Let $I$ be an injective $S$-module, let $\ib\subseteq S$ be an ideal, and let $f\colon\ib\rightarrow\Gamma_{\ia S}(I)$ be a morphism of $S$-modules. Then, $f$ induces a morphism of $R$-modules\linebreak $\ib\cap R\rightarrow\Gamma_{\ia}(I\res_R)=\Gamma_{\ia S}(I)\res_R$. The $R$-module $I\res_R$ is injective by \cite[3.6A]{lam2}, hence so is $\Gamma_{\ia}(I\res_R)$ by our hypothesis on $R$. So, there exists $y\in\Gamma_{\ia}(I)$ such that if $x\in\ib\cap R$, then $f(h(x))=xy$. Let $x\in\ib$. Since $h$ is a flat epimorphism, $\ib=(\ib\cap R)S$ by \cite[IV.2.1]{lazard}, hence $x$ is an $S$-linear combination of elements of $\ib\cap R$, and thus $f(x)=xy$. Now, $\Gamma_{\ia S}(I)$ is injective by Baer's criterion, and the claim is proven.
\end{proof}

\begin{cor}\label{1.70}
Let $h\colon R\rightarrow S$ be a flat epimorphism of rings. If $R$ has ITI with respect to every ideal (of finite type), then $S$ has ITI with respect to every ideal (of finite type).
\end{cor}

\begin{proof}
Immediately by \cite[IV.2.1]{lazard} and \ref{1.60}.
\end{proof}

\begin{cor}\label{1.80}
Let $T\subseteq R$ be a subset. If $R$ has ITI with respect to $\ia$, then $T^{-1}R$ has ITI with respect to $T^{-1}\ia$; if $R$ has ITI with respect to every ideal (of finite type), then $T^{-1}R$ has ITI with respect to every ideal (of finite type).
\end{cor}

\begin{proof}
Immediately by \ref{1.60} and \ref{1.70}.
\end{proof}

\begin{qus}\label{1.90}
Since ITI properties are inherited by rings of fractions, it is natural to ask for a converse, i.e., whether they can be delocalised in one of the following ways:
\begin{aufz}
\item[(B)]{\it If $R_{\ip}$ has ITI with respect to $\ia_{\ip}$ for every $\ip\in\spec(R)$, does then $R$ have ITI with respect to $\ia$?}
\item[(C)]{\it Let $(f_i)_{i\in I}$ be a generating family of the ideal $R$. If $R_{f_i}$ has ITI with respect to $\ia_{f_i}$ for every $i\in I$, does then $R$ have ITI with respect to $\ia$?}
\end{aufz}
The following special case of Question (B) deserves to be mentioned separately.
\begin{aufz}
\item[(D)] {\it Do pointwise noetherian rings have ITI with respect to every\linebreak ideal?\/\footnote{A ring $R$ is {\it pointwise noetherian}\/ if $R_{\ip}$ is noetherian for every $\ip\in\spec(R)$. For facts on pointwise noetherian rings, including examples of non-noetherian but pointwise noetherian rings, we suggest \cite{heinzer-ohm}.}}\hfill$\bullet$
\end{aufz}
\end{qus}

Now we turn to a first class of rings with ITI properties, namely, the absolutely flat ones. Recall that $R$ is {\it absolutely flat}\/ if it fulfills the following equivalent conditions (\cite[3.71; 4.21]{lam2}): (i) Every $R$-module is flat; (ii) Every monogeneous $R$-module is flat;\linebreak (iii) $R_{\im}$ is a field for every maximal ideal $\im\subseteq R$; (iv) $R$ is reduced and $\dim(R)\leq 0$;\linebreak (v) $R$ is von Neumann regular. Since products of infinitely many fields are absolutely flat and non-noetherian, this provides a class of interesting non-noetherian rings with ITI. (In view of Question (D) in \ref{1.90} we may note that absolutely flat rings are pointwise noetherian.)

\begin{lemma}\label{1.100}
If the $R$-module $R/\ia$ is flat, then $\ia$ is idempotent.
\end{lemma}

\begin{proof}
Applying the exact functor $\bullet\otimes_RR/\ia$ to the exact sequence $0\rightarrow\ia\hookrightarrow R\rightarrow R/\ia$ yields an exact sequence $0\rightarrow\ia/\ia^2\rightarrow R/\ia\overset{h}\rightarrow R/\ia$ with $h=\Id_{R/\ia}$. So, we get $\ia=\ia^2$ and thus the claim.
\end{proof}

\begin{prop}\label{1.110}
If the $R$-module $R/\ia$ is flat, then $R$ has ITI with respect to $\ia$.
\end{prop}

\begin{proof}
Let $M$ be an $\ia$-torsion $R$-module. The ideal $\ia$ is idempotent by \ref{1.100}, hence $\ia M=0$, and thus we can consider $M$ canonically as an $R/\ia$-module. Now, we consider the injective $R/\ia$-module $E_{R/\ia}(M)$ and its scalar restriction $E_{R/\ia}(M)\res_R$, which is an $\ia$-torsion module. Moreover, the canonical morphism $R\rightarrow R/\ia$ being flat, $E_{R/\ia}(M)\res_R$ is injective by \cite[3.6A]{lam2}. Scalar restriction to $R$ of $e_{R/\ia}(M)\colon M\rightarrow E_{R/\ia}(M)$ being a monomorphism of $R$-modules $M\rightarrow E_{R/\ia}(M)\res_R$ we thus get a monomorphism of $R$-modules $E_R(M)\rightarrow E_{R/\ia}(M)\res_R$. Therefore, $E_R(M)$ is an $\ia$-torsion module, and the claim is proven.
\end{proof}

\begin{cor}\label{1.120}
Absolutely flat rings have ITI with respect to every ideal.
\end{cor}

\begin{proof}
Immediately by \ref{1.110}.
\end{proof}

\begin{cor}\label{1.125}
Let $\im\subseteq R$ be a maximal ideal such that the $R$-module $R/\im$ is injective or that $R_{\im}$ is a field. Then, $R$ has ITI with respect to $\im$.
\end{cor}

\begin{proof}
Both conditions are equivalent by \cite[3.72]{lam2}, and they are equivalent to $R/\im$ being flat by \cite[Lemma 1]{hirano}. The claim follows then from \ref{1.110}.
\end{proof}

Next we study the relation between the property of an $R$-module $M$ of being an $\ia$-torsion module and the assassin of $M$, i.e., the set $\ass_R(M)$ of primes associated with $M$. If $R$ is noetherian, it is well-known that $M$ is an $\ia$-torsion module if and only if its assassin is contained in $\var(\ia)$ (i.e., the set of primes containing $\ia$). Since assassins are not well-behaved over non-noetherian rings, we also consider the so-called weak assassin $\assf_R(M)$. Recall that a prime ideal $\ip\subseteq R$ is {\it weakly associated with $M$} if it is a minimal prime of the annihilator of an element of $M$. For basics about weak assassins we refer the reader to \cite[IV.1 Exercise 17]{ac}, \cite{irozrush} and \cite{yassemi}.

These considerations lead to a new proof of the fact that noetherian rings have ITI with respect to every ideal, and to the result that $1$-dimensional local domains (e.g., valuation rings of height $1$) have ITI with respect to ideals of finite type.

\begin{prop}\label{1.130}
Let $M$ be an $R$-module. We consider the following statements: (1) $M$ is an $\ia$-torsion module; (2) $\assf_R(M)\subseteq\var(\ia)$; (3) $\ass_R(M)\subseteq\var(\ia)$.

We have (1)$\Rightarrow$(2)$\Rightarrow$(3); if $\ia$ is of finite type, we have (1)$\Leftrightarrow$(2)$\Rightarrow$(3); if $R$ is noetherian, we have (1)$\Leftrightarrow$(2)$\Leftrightarrow$(3).
\end{prop}

\begin{proof}
Recall that $\ass_R(M)\subseteq\assf_R(M)$, with equality if $R$ is noetherian (\cite[IV.1 Exercice 17]{ac}). Hence, (2) implies (3), and the converse holds if $R$ is noetherian.

If $M$ is an $\ia$-torsion module and $\ip\in\assf_R(M)$, there exists $x\in M$ such that $\ip$ is a minimal prime of $(0:_Rx)$, hence there exists $n\in\N$ with $\ia^nx=0$, thus we get $\ia^n\subseteq(0:_Rx)\subseteq\ip$, and therefore $\ia\subseteq\ip$. This shows (1)$\Rightarrow$(2).

Suppose now that $\ia$ is of finite type. If $\assf_R(M)\subseteq\var(\ia)$, then for $x\in M$ we have $\ia\subseteq\sqrt{(0:_Rx)}$, and thus there exists $n\in\N$ with $\ia^nx=0$. This shows (2)$\Rightarrow$(1).
\end{proof}

\begin{prop}\label{1.140}
If every $R$-module $M$ with $\ass_R(M)\subseteq\var(\ia)$ is an $\ia$-torsion module, then $R$ has ITI with respect to $\ia$.
\end{prop}

\begin{proof}
Every $R$-module has the same assassin as its injective hull (\cite[3.57]{lam2}). So, if $M$ is an $\ia$-torsion $R$-module, then $\ass_R(E_R(M))=\ass_R(M)\subseteq\var(\ia)$ by \ref{1.130}, hence $E_R(M)$ is an $\ia$-torsion module, and thus \ref{1.10} yields the claim.
\end{proof}

\begin{cor}\label{1.150}
Noetherian rings have ITI with respect to every ideal.
\end{cor}

\begin{proof}
Immediately from \ref{1.130} and \ref{1.140}.
\end{proof}

\begin{prop}\label{1.160}
If $\assf_R(M)=\assf_R(E_R(M))$ for every $R$-module $M$, then $R$ has ITI with respect to ideals of finite type.
\end{prop}

\begin{proof}
Replacing $\ass_R$ by $\assf_R$ in the proof of \ref{1.140} yields the claim.
\end{proof}

\begin{prop}\label{1.170}
Local rings whose non-maximal prime ideals are of finite type have ITI with respect to ideals of finite type.
\end{prop}

\begin{proof}
Let $R$ be such a ring and let $\im$ denote its maximal ideal. If $\im$ is of finite type then $R$ is noetherian by Cohen's theorem (\cite[0.6.4.7]{ega}) and the claim follows from \ref{1.150}. So, suppose $\im$ is not of finite type. 

Let $M$ be an $R$-module. By \ref{1.160}, it suffices to show $\assf_R(M)=\assf_R(E_R(M))$. If $M=0$, this is fulfilled. We suppose $M\neq 0$ and assume $\assf_R(M)\subsetneqq\assf_R(E_R(M))$. By \cite[1.8]{yassemi2}, a prime ideal of $R$ of finite type is associated with $M$ if and only if it is weakly associated with $M$. Therefore, our assumptions imply $\im\in\assf_R(E_R(M))\setminus\assf_R(M)$. So, there exists $x\in E_R(M)$ such that $\im$ is a minimal prime of $(0:_Rx)$, and as $M\hookrightarrow E_R(M)$ is essential there exists $r\in R$ with $rx\in M\setminus 0$. The ideal $(0:_Rrx)\subseteq R$ has a minimal prime $\ip$ which is weakly associated with $M$, so that $\ip\neq\im$, but it contains $(0:_Rx)$, yielding the contradiction $\ip=\im$.
\end{proof}

\begin{cor}\label{1.180}
$1$-dimensional local domains have ITI with respect to ideals of finite type.
\end{cor}

\begin{proof}
Immediately from \ref{1.170}.
\end{proof}

\begin{qus}\label{1.190}
The successful application of \ref{1.160} in the proof of \ref{1.170} lets us ask for more:
\begin{itemize}
\item[(E)]{\it For which rings $R$ do we have $\assf_R(M)=\assf_R(E_R(M))$ for every $R$-module $M$?}
\end{itemize}
By \cite[3.57]{lam2}, one class of such rings, containing the noetherian ones, are those over which assassins and weak assassins coincide:
\begin{itemize}
\item[(F)]{\it For which rings $R$ do we have $\ass_R(M)=\assf_R(M)$ for every $R$-module $M$?}
\end{itemize}
One may note that assassins and weak assassins do not necessarily coincide over a $1$-dimensional local domain: If $R$ is a valuation ring with value group $\Q$ and valuation $\nu$, and $\ia\dfgl\{x\in R\mid\nu(x)^2>2\}$, then $\ass_R(R/\ia)$ is empty and $\assf_R(R/\ia)$ contains the maximal ideal of $R$.\footnote{This example was suggested by Neil Epstein via {\tt MathOverflow.net}.}\hfill$\bullet$
\end{qus}

We end this section with two examples, showing that the hypotheses in \ref{1.130} cannot be omitted.

\begin{exas}\label{1.200}
A) Let $K$ be a field. We consider the absolutely flat ring $R\dfgl K^{\N}$, the ideal $\ib\subseteq R$ whose elements are precisely the elements of $R$ of finite support, the $R$-module $M\dfgl R/\ib$, and $f=(f_n)_{n\in\N}\in R$ with $f_0=0$ and $f_n=1$ for every $n>0$. Then, $f\in R\setminus 0$ is idempotent, and hence the principal ideal $\ia\dfgl\langle f\rangle_R$ is not nilpotent. Yassemi showed in \cite[Section 1, Example]{yassemi} that $\ass_R(M)=\emptyset$ and in particular $\ass_R(M)\subseteq\var(\ia)$. If $M$ is an $\ia$-torsion module and $x=(x_n)_{n\in\N}\in R$, then setting $x'_0=0$ and $x'_n=x_n$ for $n>0$ it follows $x'\dfgl(x'_n)_{n\in\N}\in\ib$, and so there exists $n_0\in\N$ such that if $n\geq n_0$, then $x_n=0$ -- a contradiction. This show that $M$ is not an $\ia$-torsion module.

In particular, the implication (3)$\Rightarrow$(1) in \ref{1.130} does not necessarily hold, even if $\ia$ is principal.

\smallskip

B) Let $R$ be a $0$-dimensional local ring whose maximal ideal $\im$ is not nilpotent, e.g. the ring constructed below in \ref{2.20}\,c). Then, $R$ is not an $\im$-torsion module, but $\assf_R(R)\subseteq\spec(R)=\var(\im)$.

In particular, the implication (2)$\Rightarrow$(1) in \ref{1.130} does not necessarily hold.\hfill$\bullet$
\end{exas}


\section{Rings without ITI}

In this section, we exhibit several examples of rings without the ITI property. The first bunch of examples relies on the following observation about local rings with idempotent maximal ideal.

\begin{prop}\label{2.10}
If $R$ is a local ring whose maximal ideal $\im$ is idempotent, then the following statements are equivalent: (i) $R$ is a field; (ii) $R$ is noetherian; (iii) $R$ has ITI with respect to $\im$; (iv) $E_R(R/\im)$ is an $\im$-torsion module.
\end{prop}

\begin{proof}
By \ref{1.30} and \ref{1.150} it suffices to show that (iv) implies (i). Suppose $E_R(R/\im)$ is an $\im$-torsion module and let $x\in E_R(R/\im)\setminus 0$. By idempotency of $\im$ we have $\im x=0$, hence $\langle x\rangle_R\cong R/\im$ is simple. The canonical injection $R/\im\hookrightarrow E_R(R/\im)$ being essential we have $(R/\im)\cap\langle x\rangle_R\neq 0$, hence, by simplicity, $x\in R/\im$. Thus, $R/\im=E_R(R/\im)$ is injective, and therefore $R$ is a field by \cite[3.72]{lam2}.
\end{proof}

Recall that a {\it valuation ring} is a local domain whose set of (principal) ideals is totally ordered with respect to inclusion, and that a {\it Bezout ring} is a domain whose ideals of finite type are principal. A quotient of a valuation ring or of a Bezout ring has the property that the set of its ideals is totally ordered with respect to inclusion or that its ideals of finite type are principal, respectively. Moreover, valuation rings are Bezout rings, and Bezout rings are coherent rings whose local rings are valuation rings (\cite[VI.1.2; VII.1 Exercice 20; VII.2 Exercices 12, 14, 17]{ac}).

\begin{prop}\label{2.20}
Let $K$ be a field, let $Q$ denote the additive monoid of positive rational numbers, let $R\dfgl K[Q]$ denote the algebra of $Q$ over $K$, let $\{e_{\alpha}\mid\alpha\in Q\}$ denote its canonical basis, and let $\im\dfgl\langle e_{\alpha}\mid\alpha>0\rangle_R$ and $\ia\dfgl\langle e_{\alpha}\mid\alpha>1\rangle_R$. Furthermore, set $S\dfgl R_{\im}$, $\inn\dfgl\im_{\im}$, $\ib\dfgl\ia_{\im}$, $T\dfgl S/\ib$ and $\ip\dfgl\inn/\ib$.

a) $R$ is a $1$-dimensional Bezout ring that does not have ITI with respect to its maximal ideal $\im$.\footnote{In \cite[X.1 Exercice 27 f)]{a} the reader gets the cruel task to show that $R$ is a valuation domain -- but its element $1+e_1$ is neither invertible nor contained in $\im$.}

b) $S$ is a $1$-dimensional valuation ring that has ITI with respect to non-maximal ideals, but does not have ITI with respect to its maximal ideal $\inn$.

c) $T$ is a $0$-dimensional non-coherent local ring that has ITI with respect to non-maximal ideals, but does not have ITI with respect to its maximal ideal $\ip$.
\end{prop}

\begin{proof}
First, we note that $\im$ is an idempotent maximal ideal of $R$ and the only prime ideal of $R$ containing $\ia$. Hence, $S$ and $T$ are local rings with idempotent maximal ideals $\inn$ and $\ip$, respectively, and $T\cong(R/\ia)_{\im/\ia}\cong R/\ia$ is $0$-dimensional. So, there is a surjective morphism of rings $R\rightarrow T$, mapping $\im$ onto $\ip$. The ring $R$ is the union of the increasing family $(K[e_{\frac{1}{n!}}])_{n\in\N}$ of principal subrings, hence a Bezout ring, and it is integral over its $1$-dimensional subring $K[e_1]$ by \cite[12.4]{gilmer}, hence $1$-dimensional (cf.~\cite[Example 31]{hutchins}). Therefore, $S$ is a $1$-dimensional valuation ring, and in particular not a field. Moreover, $\ip=(0:_T\frac{e_1}{1}+\ib)$ is not of finite type, hence $T$ is not coherent (\cite[4.60]{lam2}), and in particular not noetherian. Thus, the negative claims about ITI follow from \ref{1.40} and \ref{2.10}.

For an ideal $\ic\subseteq S$ we set $\alpha(\ic)\dfgl\inf\{\alpha\in Q\mid e_{\alpha}\in\ic\}$. We have $\alpha(\inn)=0$, and if $\ic,\id\subseteq S$ are ideals with $\alpha(\ic)>\alpha(\id)$ then $\ic\subseteq\id$. Since ideals of $S$ of finite type are of the form $\langle e_{\alpha}\rangle_S$ with $\alpha\in Q$, it follows that an ideal $\ic\subseteq S$ is of finite type if and only if $e_{\alpha(\ic)}\in\ic$. Now, let $\ic\subseteq S$ be an ideal and let $n\in\N$. If $\alpha\in Q$ with $e_{\alpha}\in\ic^n$, then there are $\alpha_1,\ldots,\alpha_n\in Q$ with $e_{\alpha_1},\ldots,e_{\alpha_n}\in\ic$ and $e_{\alpha}=e_{\alpha_1+\cdots+\alpha_n}$, implying $\alpha=\alpha_1+\cdots+\alpha_n\geq n\alpha(\ic)$. Therefore, $\alpha(\ic^n)\geq n\alpha(\ic)$. Conversely, if $\beta\in Q$ with $\beta>n\alpha(\ic)$, then $\frac{\beta}{n}>\alpha(\ic)$, implying $e_{\frac{\beta}{n}}\in\ic$ and thus $e_{\beta}\in\ic^n$. Therefore, $\alpha(\ic^n)\leq n\alpha(\ic)$. It follows $\alpha(\ic^n)=n\alpha(\ic)$.

Next, let $\ic,\id\subseteq S$ be non-maximal, non-zero ideals and let $n\in\N$, so that $\alpha(\ic^n)\neq 0$ and $\alpha(\id^n)\neq 0$. There exist $k,m\in\N$ with $\alpha(\ic^{nk})=k\alpha(\ic^n)\geq\alpha(\id^n)$ and $\alpha(\id^{nm})=m\alpha(\id^n)\geq\alpha(\ic^n)$, implying $\ic^{nk}\subseteq\id^n$ and $\id^{nm}\subseteq\ic^n$, and therefore $\Gamma_{\ic}=\Gamma_{\id}$ (cf. \cite[2.2]{r-torsion}). It follows that $S$ has ITI with respect to non-maximal ideals if and only if it has ITI with respect to some non-maximal ideal, and by \ref{1.180} this is fulfilled.

Finally, if $\ic\subseteq S$ is a non-maximal ideal then $\alpha(\ic)>0$, hence there exists $n\in\N$ with $\alpha(\ic^n)=n\alpha(\ic)>1$ and therefore $\ic^n\subseteq\ib$. Thus, non-maximal ideals of $T$ are nilpotent, so that $T$ has ITI with respect to non-maximal ideals by \ref{1.20} A).
\end{proof}

We draw some conclusions from these examples. While a $1$-dimensio\-nal local domain has ITI with respect to ideals of finite type (\ref{1.180}), it need not have so with respect to every ideal (\ref{2.20} b)). While a reduced $0$-dimensional ring (i.e., an absolutely flat ring) has ITI with respect to every ideal (\ref{1.120}), an arbitrary $0$-dimensional ring need not have so (\ref{2.20}~c)). John Greenlees conjectured (personal communication, 2010) that a ring with bounded $\ia$-torsion\footnote{A ring $R$ is {\it of bounded $\ia$-torsion} for some ideal $\ia$ if there exists $n\in\N$ with $\Gamma_{\ia}(R)=(0:_R\ia^n)$.} for some ideal $\ia$ has ITI with respect to $\ia$; in particular, domains would have ITI with respect to every ideal. Example \ref{2.20} b) disproves this conjecture.

\begin{cor}\label{2.30}
The polynomial algebra $K[(X_i)_{i\in I}]$ in infinitely ma\-ny indeterminates $(X_i)_{i\in I}$ over a field $K$ does not have ITI with respect to $\langle X_i\mid i\in I\rangle_{K[(X_i)_{i\in I}]}$.
\end{cor}

\begin{proof}
Let $U\dfgl K[(X_i)_{i\in I}]$ and let $\iq\dfgl\langle X_i\mid i\in I\rangle_U$. Assume that $U$ has ITI with respect to $\iq$. Then, $U_\iq$ has ITI with respect to $\iq_\iq$ by \ref{1.80}. Furthermore, in the notations of \ref{2.20}, there is a surjective morphism of rings $U\rightarrow K[Q]$ mapping $\iq$ onto $\im$. This morphism induces a surjective morphism of rings $U_\iq\rightarrow S$ mapping $\iq_\iq$ onto $\inn$. Now, \ref{1.30} and \ref{1.40} imply that $E_S(S/\inn)$ is an $\inn$-torsion module, and thus $S$ has ITI with respect to $\inn$ by \ref{2.10}. This contradicts \ref{2.20} b).
\end{proof}

One may note that $R=K[(X_i)_{i\in\N}]$ is not only a coherent domain, but in fact a filtering inductive limit of noetherian rings with flat monomorphisms (cf. \cite[p. 48]{glaz}), and that its $\langle X_i\mid i\in\N\rangle_R$-adic topology is separated. (In \ref{2.110} we will give a direct proof of the fact that $E_R(R/\langle X_i\mid i\in\N\rangle_R)$ is not a $\langle X_i\mid i\in\N\rangle_R$-torsion module.)

\begin{qu}\label{2.40}
Inspired by \ref{2.30} we ask the following question:
\begin{itemize}
\item[(G)] {\it Do polynomial algebras in (countably) infinitely many indeterminates over fields have ITI with respect to (monomial) ideals of finite type?}\hfill$\bullet$
\end{itemize}
\end{qu}

As an application of \ref{2.10}, we show that ITI is not necessarily preserved by tensor products. In particular, it is not stable under base change.

\begin{prop}\label{2.50}
Let $p$ be a prime number, let $K$ be a non-perfect field of characteristic $p$, and let $L$ be its perfect closure. Then, $L\otimes_KL$ has ITI with respect to ideals of finite type, but does not have ITI with respect to its maximal ideal\/ $\nil(L\otimes_KL)$.\footnote{This is also the standard example showing that noetherianness is not necessarily preserved by tensor products (cf. \cite[I.3.2.5]{ega}).}
\end{prop}

\begin{proof}
It is well-known that $L\otimes_KL$ is a non-noetherian $0$-dimensional local ring with maximal ideal $\nil(L\otimes_KL)$ (\cite[I.3.2.5]{ega}). So, the first claim is clear by \ref{1.20} B), and by \ref{2.10}, it suffices to show that $\nil(L\otimes_KL)$ is idempotent. Let $f\in\nil(L\otimes_KL)\setminus 0$. There exist $m\in\N$ and families $(x_i)_{i=0}^m$ and $(y_i)_{i=0}^m$ in $L\setminus 0$ with $f=\sum_{i=0}^mx_i\otimes y_i$. Since $L$ is perfect, there exist for every $i\in[0,m]$ elements $v_i,w_i\in L$ with $v_i^p=x_i$ and $w_i^p=y_i$. It follows $f=\sum_{i=0}^mv_i^p\otimes w_i^p=(\sum_{i=0}^mv_i\otimes w_i)^p$ and therefore $f\in\nil(L\otimes_KL)^p$ as desired.
\end{proof}

Up to now, we saw only rings without ITI with respect to ideals not of finite type. The second bunch of examples relies on an observation about valuation rings with maximal ideal of finite type (\ref{2.80}). 

\begin{lemma}\label{2.60}
Let $\ia,\ib\subseteq R$ be ideals and suppose the canonical injection $\ib\hookrightarrow R$ is essential. Then, $E_R(\ib)$ is an $\ia$-torsion module if and only if\/ $\ia$ is nilpotent.
\end{lemma}

\begin{proof}
As $\ib\hookrightarrow R$ is essential, $R\subseteq E_R(R)=E_R(\ib)$. If $E_R(\ib)$ is a $\ia$-torsion module, then so is $R$, hence $\ia$ is nilpotent. The converse is clear.
\end{proof}

\begin{cor}\label{2.70}
Let $R$ be a quotient of a valuation ring, let $x\in R\setminus 0$, and let $\ia\subseteq R$ be an ideal. Then, $E_R(\langle x\rangle_R)$ is an $\ia$-torsion module if and only if $\ia$ is nilpotent.
\end{cor}

\begin{proof}
The set of principal ideals of $R$ is totally ordered with respect to inclusion, hence the canonical injection $\langle x\rangle_R\hookrightarrow R$ is essential, and thus \ref{2.60} yields the claim.
\end{proof}

\begin{prop}\label{2.80}
If $R$ is a valuation ring whose maximal ideal $\im$ is of finite type, then the following statements are equivalent: (i) $R$ is noetherian; (ii) $R$ has ITI with respect to $\im$; (iii) $E_R(R/\im)$ is an $\im$-torsion module.
\end{prop}

\begin{proof}
By \ref{1.30} and \ref{1.150}, it suffices to show that (iii) implies (i). Suppose (iii) holds and assume $R$ is non-noetherian. By \cite[VI.3.6 Proposition 9]{ac}, a valuation ring with maximal ideal $\im$ of finite type is noetherian if and only if its $\im$-adic topology is separated. Hence, there exists $x\in(\bigcap_{n\in\N}\im^n)\setminus 0$. We consider the local ring $\overline{R}\dfgl R/x\im$, its maximal ideal $\overline{\im}\dfgl\im/x\im$, and $\overline{x}\dfgl x+x\im\in\overline{R}$. If there exists $n\in\N$ with $\overline{\im}^n=0$, then $\im^n\subseteq x\im\subseteq\im^{n+1}$, and Nakayama's Lemma yields the contradiction $\im^n=0$. As $(0:_{\overline{R}}\overline{x})=\overline{\im}$ and therefore $\overline{R}/\overline{\im}\cong\langle\overline{x}\rangle_{\overline{R}}$, we get a contradiction from \ref{1.40} and \ref{2.70}.
\end{proof}

\begin{prop}\label{2.90}
Let $p$ be a prime number, let $R\dfgl\Z[(X_i)_{i\in\N}]$ be the polynomial algebra in indeterminates $(X_i)_{i\in\N}$ over $\Z$, let $\ia\dfgl\langle p^{j-i}X_j-X_i\mid i,j\in\N, i<j\rangle_R$, let $\im\dfgl\langle p\rangle_R+\langle X_i\mid i\in\N\rangle_R$, let $S\dfgl R/\ia$, and let $\inn\dfgl\im/\ia$. Furthermore, denote by $Y_i$ the canonical image of $X_i$ in $S$ for $i\in\N$, let $T\dfgl S_{\inn}/\langle pY_0\rangle_{S_{\inn}}$, and let $\ip\dfgl\inn_{\inn}/\langle pY_0\rangle_{S_{\inn}}$.

a) $S_{\inn}$ is a $2$-dimensional valuation ring that does not have ITI with respect to its principal maximal ideal $\inn_{\inn}$.

b) $T$ is a $1$-dimensional local ring that does not have ITI with respect to its principal maximal ideal $\ip$.
\end{prop}

\begin{proof}
We claim first that if $i,m,n,r,s\in\N$, then $p^mY_i^r$ and $p^nY_i^s$ are comparable with respect to divisibility in $S$. Indeed, without loss of generality, we can suppose $r<s$, so if $m\leq n$, then the claim is clear. Otherwise, $p^nY_i^s=p^mY_i^rY_i^{s-1-r}Y_{i+m-n}$ and thus the claim holds. Keeping in mind the definition of $\ia$, we see that if $b\in\N$ and $f\in S$, then there exist $i\in\N$ with $i\geq b$ and a finite family $(a_k)_{k=0}^r$ in $\Z$ with $f=\sum_{k=0}^ra_kY_i^k$. In particular, if $f,g\in S$, then there exist $i\in\N$ and finite families $(a_k)_{k=0}^r$ and $(b_k)_{k=0}^s$ in $\Z$ with $f=\sum_{k=0}^ra_kY_i^k$ and $g=\sum_{k=0}^sb_kY_i^k$. Moreover, for $f=\sum_{k=0}^ra_kp^{n_k}Y_i^k\in S\setminus 0$ with $i\in\N$ and finite families $(n_k)_{k=0}^r$ in $\N$ and $(a_k)_{k=0}^r$ in $\Z\setminus\langle p\rangle$ there exists $k_0\dfgl$\linebreak$\min\{k\in[0,r]\mid a_k\neq 0\}$, and $p^{n_{k_0}}Y_i^{k_0}$ divides $p^{n_k}Y_i^k$ for every $k\in[0,r]$ with $a_k\neq 0$. In particular, for $f\in S_{\inn}\setminus 0$ there exist a unit $u$ of $S_{\inn}$ and $i,k,n\in\N$ with $f=up^nY_i^k$, and we have $f\in\inn_{\inn}$ if and only if $(n,k)\neq(0,0)$. It follows that any two elements of $S_{\inn}$ are comparable with respect to divisibility.

Next, we show that if $a\in\Z\setminus 0$ and $n\in\N$, then $aX_0^n\notin\ia$. Assume $aX_0^n\in\ia$. There exists $m\in\N$ with $aX_0^n\in\ia'\dfgl\langle p^{j-i}X_j-X_i\mid i,j\in[0,m], i<j\rangle_{R'}$, where $R'\dfgl\Z[(X_i)_{i=0}^m]$. If $i,j\in[0,m]$ with $i<j$, then $p^{j-i}X_j-X_i=(p^{m-i}X_m-X_i)-p^{j-i}(p^{m-j}X_m-X_j)$, hence $\ia'=\langle p^{m-i}X_m-X_i\mid i\in[0,m-1]\rangle_{R'}$. Factoring out $\langle p^{m-i}X_m-X_i\mid i\in[1,m-1]\rangle_{R'}$ we get $aU^n=f(p^mV-U)$ in the polynomial algebra $\Z[U,V]$ with $f\in\Z[U,V]\setminus 0$, a contradiction, and thus our claim. From this we get $pX_0\notin\ia$ and thus $\ic\dfgl\langle X_i\mid i\in\N\rangle_R/\ia\neq 0$. Moreover, it also follows that $S$ is a domain. Indeed, if $f,g\in S\setminus 0$ with $fg=0$, then there exist $i,k,l,m,n\in\N$, $a,b\in\Z\setminus\langle p\rangle$, and $f',g'\in\ic$ with $f=p^mY_i^k(a+f')$ and $g=p^nY_i^l(b+g')$. Furnishing $R$ with its canonical $\Z$-graduation and $S$ with the induced $\Z$-graduation, we get $abp^{m+n}Y_i^{k+l}=0$, hence $abp^rX_0^s\in\ia$ for some $r,s\in\N$, and finally the contradiction $ab=0$. So, $S$ is a domain, and hence $S_{\inn}$ is a valuation ring.

If $i\in\N$ then $Y_i=pY_{i+1}$, hence $p$ divides $Y_i$ in $S$, and therefore $\inn_{\inn}=$\linebreak$\langle p\rangle_{S_{\inn}}+\langle Y_i\mid i\in\N\rangle_{S_{\inn}}=\langle p\rangle_{S_{\inn}}$ is principal. Thus, $\ip=\langle p\rangle_T$ is principal, too. Now, let $\iq\in\spec(S_{\inn})\setminus\{0,\inn_{\inn}\}$, so that $p\notin\iq$. Then, there exists an $i\in\N$ with $Y_i\in\iq$, since elements of $\inn_{\inn}$ are of the form $up^nY_i^k$ with a unit $u$ of $S_{\inn}$, $i,k,n\in\N$ and $(n,k)\neq(0,0)$. This implies $p^nY_{i+n}=Y_i\in\iq$ and thus $Y_{i+n}\in\iq$ for every $n\in\N$, and also $Y_{i-k}=p^kY_i\in\iq$ for every $k\in[0,i]$. It follows $\iq=\ic_\inn$, hence $\spec(S_{\inn})=\{0,\ic_{\inn},\inn_{\inn}\}$, and therefore $\dim(S_{\inn})=2$. Furthermore, if $i,n\in\N$ then $Y_i=p^nY_{i+n}$, hence $p^n$ divides $Y_i$ in $S$, and therefore $0\neq pY_0\in\ic_{\inn}\subseteq\bigcap_{n\in\N}\inn_{\inn}^n$. Thus, $T$ is $1$-dimensional and $S_{\inn}$ is non-noetherian, so a) follows from \ref{2.80}.

If $n\in\N^*$, then $p^n$ divides $Y_0=p^nY_n$, so if $pY_0$ divides $p^n$, then $Y_n$ is a unit of $S_{\inn}$, which is a contradiction. Therefore, $\langle pY_0\rangle_{S_{\inn}}$ contains no power of $p$, and hence $\ip$ is not nilpotent. Moreover, $Y_0\notin\langle pY_0\rangle_{S_{\inn}}$, so denoting by $Z$ the canonical image of $Y_0$ in $T$ it follows $\ip=(0:_TZ)$, hence $T/\ip\cong\langle Z\rangle_T$, and thus \ref{1.30} and \ref{2.70} imply that $T$ does not have ITI with respect to $\ip$.
\end{proof}

Consider again the ring $T$ from \ref{2.90}. By \ref{1.160} there exists a $T$-module $M$ with $\assf_T(M)\neq\assf_T(E_T(M))$, and we can explicitly describe such a module. Indeed, $T$ has precisely one non-maximal prime and $E_T(T/\ip)$ is not a $\ip$-torsion module, so \ref{1.130} implies $\{\ip\}=\assf_T(T/\ip)\subsetneqq\assf_T(E_T(T/\ip))=\spec(T)$.\smallskip

We draw from the examples in \ref{2.20} and \ref{2.90} the conclusion that while  a $1$-dimensio\-nal local domain has ITI with respect to ideals of finite type (\ref{1.180}), an arbitrary  $1$-dimensional local ring need not have so. Further examples of non-noetherian valuation rings whose maximal ideal is of finite type can be found in \cite[p.~79, Remark]{mat} and \cite[Example 32]{hutchins}.\smallskip

Based on \ref{2.90} we show now that the idealisation of a module over a noetherian ring need not have ITI. (Recall that for a ring $R$ and an $R$-module $M$, the idealisation of $M$ is the $R$-algebra with underlying $R$-module $R\oplus M$ and with multiplication defined by $(r,x)\cdot(s,y)=(rs,ry+sx)$.)

\begin{prop}\label{2.100}
We use the notations from \ref{2.90} and denote by $U$ the idealisation of the $\Z_{\langle p\rangle}$-module obtained from $\langle Y_i\mid i\in\N\rangle_{S_{\inn}}/\langle pY_0\rangle_{S_{\inn}}$ by scalar restriction. Then, $U$ is a local ring that does not have ITI with respect to its principal maximal ideal.
\end{prop}

\begin{proof}
For $i\in\N$ we denote by $Z_i$ the canonical image of $Y_i$ in $M\dfgl\langle Y_i\mid i\in\N\rangle_{S_{\inn}}/\langle pY_0\rangle_{S_{\inn}}$. If $i,j\in\N$, then $Y_iY_j=Y_0Y_{i+j}=pY_0Y_{i+j+1}$ in $S_{\inn}$. Therefore, an element $m$ of $M\setminus 0$ has the form $m=\sum_{j=1}^kv_jZ_{i_j}$ with a strictly increasing sequence $(i_j)_{j=1}^k$ in $\N$ and a family $(v_j)_{j=1}^k$ in $\Z_{\langle p\rangle}\setminus\{0\}$. Since $Z_{i_j}=p^{i_k-i_j}Z_{i_k}$ for $j\in[1,k-1]$, it follows that $m$ can be written in the form $m=vZ_i$ with $i\in\N$ and $v\in\Z_{\langle p\rangle}\setminus\{0\}$. Suppose that $i$ is chosen minimally with this property, and assume $v\in\langle p\rangle_{\Z_{\langle p\rangle}}$. Then, $v=v'p^n$ with $v'\in\Z_{\langle p\rangle}\setminus\langle p\rangle$ and $n>0$. If $n>i$, we get the contradiction $vZ_i=v'p^nZ_i=v'p^{n-i-1}pZ_0=0$. If $n\leq i$, we get $vZ_i=v' p^nZ_i=v' Z_{i-n}$, contradicting the minimality of $i$. Therefore, $v\in\Z_{\langle p\rangle}\setminus\langle p\rangle$.

The above shows that an element of $U\setminus 0$ has the form $(up^n,vZ_i)$ with $i,n\in\N$ and $u,v\in(\Z_{\langle p\rangle}\setminus\langle p\rangle)\cup\{0\}$ such that $(u,v)\neq(0,0)$. Using this, it is readily checked that $U$ is a local ring with maximal ideal $\iq=\langle(p,0)\rangle_U$, that the morphism of $U$-modules $U\rightarrow U$ with $1\mapsto(0,Z_0)$ induces an isomorphism of $U$-modules $U/\iq\cong\langle(0,Z_0)\rangle_U$, and that $\bigcap_{n\in\N}\iq^n=\bigcap_{n\in\N}\langle(p^n,0)\rangle_U=\bigcap_{n\in\N}(\langle p^n\rangle\oplus M)=0\oplus M=\langle(0,Z_i)\mid i\in\N\rangle_U\neq 0$. In particular, $\iq$ is not nilpotent. Now, we consider $(up^n,vZ_i)\in U\setminus 0$ with $i,n\in\N$ and $u,v\in(\Z_{\langle p\rangle}\setminus\langle p\rangle)\cup\{0\}$. If $u\neq 0$, then $(up^n,vZ_i)(0,u^{-1}Z_n)=(0,Z_0)$, and if $u=0$, then $v\neq 0$, and therefore $(up^n,vZ_i)(p^i,0)=(0,Z_0)$. From this it follows that the canonical injection $\langle(0,Z_0)\rangle_U\hookrightarrow U$ is essential, thus \ref{1.30} and \ref{2.60} yield the claim.
\end{proof}

We end this section by considering again a $0$-dimensional local ring $R$ with maximal ideal $\im$. If $\im$ is nilpotent, then $R$ has ITI with respect to every ideal (\ref{1.20} A)). If $\im$ is idempotent, then $R$ has ITI with respect to every ideal if and only if it is a field (\ref{2.10}). All examples of $0$-dimensional local rings without ITI constructed so far had an idempotent maximal ideal (\ref{2.20}, \ref{2.50}). We present now a further $0$-dimensional local ring $R$; its maximal ideal $\im$ is neither nilpotent nor idempotent, but T-nilpotent\footnote{An ideal $\ia$ of $R$ is {\it T-nilpotent}\/ if for every family $(x_i)_{i\in\N}$ in $\ia$ there exists $n\in\N$ with $\prod_{i=0}^nx_i=0$.}, and its $\im$-adic topology is separated. But still $R$ does not have ITI with respect to $\im$.

\begin{lemma}\label{2.110}
Let $K$ be a field, let $A\dfgl K[(X_i)_{i\in\N}]$ be the polynomial algebra in countably infinitely many indeterminates $(X_i)_{i\in\N}$, let $\inn\dfgl\langle X_i\mid i\in\N\rangle_A$ and let $\ia\dfgl\langle\{X_iX_j\mid i,j\in\N,i\neq j\}\cup\{X_i^{i+1}\mid i\in\N\}\rangle_A$. Then, there exists $f\in E_A(A/\inn)\setminus\Gamma_{\inn}(E_A(A/\inn))$ with $\ia f=0$.
\end{lemma}

\begin{proof}
Let $\mathbbm{M}$ denote the set of monomials in $A$. We furnish $E\dfgl\hm{K}{A}{K}$ with its canonical structure of $A$-module. Its elements are families in $K$ indexed by $\mathbbm{M}$. For $g=(\alpha_t)_{t\in\mathbbm{M}}\in E$ and $s\in\mathbbm{M}$ we have $sg=(\alpha_{st})_{t\in\mathbbm{M}}$. The morphism of $A$-modules $A\rightarrow E$ mapping $1$ to $(\alpha_t)_{t\in\mathbbm{M}}$ with $\alpha_1=1$ and $\alpha_t=0$ for $t\neq 1$ has kernel $\inn$ and hence induces a monomorphism of $A$-modules $A/\inn\rightarrowtail E$ by means of which we consider $A/\inn$ as a sub-$A$-module of $E$.

Let $f=(\alpha_t)_{t\in\mathbbm{M}}$ with $\alpha_t=1$ for $t\in\{X_i^i\mid i\in\N\}$ and $\alpha_t=0$ for $t\notin\{X_i^i\mid i\in\N\}$. If $i,j\in\N$ with $i\neq j$ and $t\in\mathbbm{M}$, then $X_iX_jf(t)=f(X_iX_jt)=0$. If $i\in\N$ and $t\in\mathbbm{M}$, then $X_i^{i+1}f(t)=f(X_i^{i+1}t)=0$. This implies $\ia f=0$. If $n\in\N$, then $X_n^n\in\inn^n$ and $X_n^nf(1)=f(X_n^n)=1$. Thus, if $n\in\N$, then $\inn^n f\neq 0$. Furthermore, $X_1f(1)=f(X_1^1)=1$, and for $t\in\mathbbm{M}\setminus\{1\}$ we have $X_1f(t)=f(X_1t)=0$, so that $X_1f$ is a non-zero element of $A/\inn$. Therefore, the canonical injection $A/\inn\hookrightarrow A/\inn+\langle f\rangle_A$ is essential, and so it follows $f\in E_A(A/\inn)$ as desired.
\end{proof}

\begin{prop}\label{2.120}
Let $K$ be a field, let $A$, $\inn$ and $\ia$ be as in \ref{2.110}, let $R\dfgl A/\ia$, denote for $i\in\N$ by $Y_i$ the canonical image of $X_i$ in $R$, and let $\im\dfgl\langle Y_i\mid i\in\N\rangle_R$. Then, $R$ is a $0$-dimensional local ring; its maximal ideal $\im$ is neither nilpotent nor idempotent, but $T$-nilpotent; its $\im$-adic topology is separated; the $\im$-torsion functor $\Gamma_{\im}$ is not a radical\/\footnote{A {\it radical}\/ is a subfunctor $F$ of $\Id_{\catmod(R)}$ with $F(M/F(M))=0$ for every $R$-module $M$.}; $R$ does not have ITI with respect to $\im$.
\end{prop}

\begin{proof}
We set $P\dfgl\{X_iX_j\mid i,j\in\N,i\neq j\}$ and $Q\dfgl\{X_i^{i+1}\mid i\in\N\}$. The ideal $\im$ is maximal since $R/\im\cong K$. As $Y_i$ is nilpotent for every $i\in\N$ it follows $\im\subseteq\nil(R)\subseteq\im$, hence $\im=\nil(R)$. Thus, $R$ is a $0$-dimensional local ring with maximal ideal $\im$. If $n\in\N$ and $\im^n=0$, then $Y_n^n=0$, hence $X_n^n\in\ia$ -- a contradiction; therefore, $\im$ is not nilpotent. If $\im$ is idempotent, then $Y_1\in\im^2=\langle Y_i^2\mid i>1\rangle_R$, hence $X_1$ is a polynomial in $P\cup Q\cup\{X_i^2\mid i>1\}$, thus a polynomial in $P\cup\{X_0\}\cup\{X_i^2\mid i>0\}$ -- a contradiction; therefore, $\im$ is not idempotent. Assume now there is a family $(f_i)_{i\in\N}$ in $\im$ with $\prod_{i=0}^nf_i\neq 0$ for every $n\in\N$. Without loss of generality, we can suppose all the $f_i$ are monomials in $\{Y_j\mid j\in\N\}$. As $Y_iY_j=0$ for $i,j\in\N$ with $i\neq j$, there exists $k\in\N$ such that all the $f_i$ are monomials in $Y_k$, yielding the contradiction $\prod_{i=0}^{k+1}f_i=0$. Thus, $\im$ is T-nilpotent. If $n\in\N$, then $\im^n=\langle Y_i^n\mid i\geq n\rangle_R$, so in elements of $\im^n$ there occurs no $Y_i$ with $i<n$. It follows that in elements of $\bigcap_{n\in\N}\im^n$ there occurs no $Y_i$ at all, hence $\bigcap_{n\in\N}\im^n=0$ and thus $R$ is $\im$-adically separated. It is readily seen that $1+\Gamma_{\im}(R)$ is a non-zero element of $\Gamma_{\im}(R/\Gamma_{\im}(R))$, hence $\Gamma_{\im}$ is not a radical. Finally, the $A$-module $(0:_{E_A(A/\inn)}\ia)$ is not an $\inn$-torsion module by \ref{2.110}, hence by base ring independence of torsion functors and \cite[10.1.16]{bs} we see that $E_R(R/\im)\cong(0:_{E_A(A/\inn)}\ia)$ is not an $\im$-torsion module. Thus, \ref{1.30} implies that $R$ does not have ITI with respect to $\im$.
\end{proof}

If, in \ref{2.110} and \ref{2.120}, we take $K$ instead of a field to be a selfinjective ring, then the conclusions still hold, except that $R$ need not be a $0$-dimensional local ring and that $\im$ need not be a maximal ideal of $R$.


\section{Weak proregularity, and applications to local cohomology}

In this section we clarify the relation between ITI and weak proregularity, and then sketch some basic results on local cohomology for rings with ITI, but omit proofs. Results under ITI hypotheses on the closely related higher ideal transformation functors can be found in \cite[2.3]{r-sheaves}.

\smallskip

Let $R$ be a ring and let $\ia\subseteq R$ be an ideal. The right derived cohomological functor of the $\ia$-torsion functor $\Gamma_{\ia}$ is denoted by $(H_{\ia}^i)_{i\in\Z}$, and $H_{\ia}^i$ is called {\it the $i$-th local cohomology functor with respect to $\ia$.} There is a canonical isomorphism of functors $\Gamma_{\ia}(\bullet)\cong\ilim_{n\in\N}\hm{R}{R/\ia^n}{\bullet}$ that can be canonically extended to an isomorphism of $\delta$-functors $(H_{\ia}^i(\bullet))_{i\in\Z}\cong(\ilim_{n\in\N}\ext{i}{R}{R/\ia^n}{\bullet})_{i\in\Z}$ (\cite[1.3.8]{bs}).

\smallskip

Suppose now that $\ia$ is of finite type and let ${\bf a}=(a_i)_{i=1}^n$ be a generating family of $\ia$. \v{C}ech cohomology with respect to ${\bf a}$ yields an  exact $\delta$-functor, denoted by $(\check{H}^i({\bf a},\bullet))_{i\in\Z}$. There is a canonical isomorphism $\Gamma_{\ia}(\bullet)\cong\check{H}^0({\bf a},\bullet)$ that can be canonically extended to a morphism of $\delta$-functors $\gamma_{{\bf a}}\colon(H_{\ia}^i(\bullet))_{i\in\Z}\rightarrow(\check{H}^i({\bf a},\bullet))_{i\in\Z}$ (\cite[5.1]{bs}). The sequence ${\bf a}$ is called \textit{weakly proregular} if $\gamma_{{\bf a}}$ is an isomorphism. (By \cite[3.2]{schenzel} this is equivalent to the usual -- but more technical -- definition of weak proregularity (\cite[Corrections]{lipman}, \cite[4.21]{yekutieli}, cf.\,\cite[Expos\'e II, Lemme 9]{sga2}).) The ideal $\ia$ is called \textit{weakly proregular} if it has a weakly proregular generating family. By \cite[6.3]{yekutieli}, this is the case if and only if every finite generating family of $\ia$ is weakly proregular. Ideals in noetherian rings are weakly proregular (\cite[5.1.20]{bs}), but the converse need not hold -- see below for examples. Local cohomology with respect to a weakly proregular ideal $\ia$ of finite type behaves quite well, and so we may ask about the relation between this notion and ITI with respect to $\ia$. The next two results clarify this.

\begin{prop}\label{wpr10}
There exist a ring $R$ and a weakly proregular ideal $\ia$ such that $R$ does not have ITI with respect to $\ia$.
\end{prop}

\begin{proof}
The $2$-dimensional domain $S_{\inn}$ from \ref{2.90} does not have ITI with respect to its principal maximal ideal $\inn_{\inn}=\langle p\rangle_{S_{\inn}}$. But $p$ is regular, and hence $\inn_{\inn}$ is weakly proregular by \cite[4.22]{yekutieli}.
\end{proof}

The proof of the next result makes use of Koszul homology and cohomology. We briefly recall our notations and refer the reader to \cite[X.9]{a} and \cite[5.2]{bs} for details. Let ${\bf a}=(a_i)_{i=1}^n$ be a finite sequence in $R$. We denote by $K_\bullet({\bf a})$ the Koszul complex with respect to ${\bf a}$ and by $K^\bullet({\bf a})\dfgl\hm{R}{K_\bullet({\bf a})}{R}$ the Koszul cocomplex with respect to ${\bf a}$. We define functors $K_\bullet({\bf a},\sq)\dfgl K_\bullet({\bf a})\otimes_R\sq$ and $K^\bullet({\bf a},\sq)\dfgl\hm{R}{K_\bullet({\bf a})}{\sq}$, and for $i\in\Z$ we set $H_i({\bf a},\sq)\dfgl H_i(K_\bullet({\bf a},\sq))$ and $H^i({\bf a},\sq)\dfgl H^i(K^\bullet({\bf a},\sq))$. For $u\in\N$ we set ${\bf a}^u=(a_i^u)_{i=1}^n$. For $u,v\in\N$ with $u\leq v$ there is a morphism of functors $K_\bullet({\bf a}^u,\sq)\rightarrow K_\bullet({\bf a}^v,\sq)$, and these morphisms give rise to an inductive system of functors $(K_\bullet({\bf a}^u,\sq))_{u\in\N}$. We denote its inductive limit by $K_\bullet({\bf a}^\infty,\sq)$, and we set $H_i({\bf a}^\infty,\sq)\dfgl H_i(K_\bullet({\bf a}^\infty,\sq))$.

\begin{prop}\label{wpr20}
Let $R$ be a ring and let $\ia\subseteq R$ be an ideal of finite type. If $R$ has ITI with respect to $\ia$, then $\ia$ is weakly proregular.
\end{prop}

\begin{proof}
Let ${\bf a}=(a_i)_{i=1}^n$ be a finite generating family of $\ia$. Let $i\in\Z$. It follows from \cite[5.2.5]{bs} that there is a canonical isomorphism of functors \[\check{H}^i({\bf a},\bullet)\cong H_{n-i}({\bf a}^\infty,\bullet).\tag{1}\] Since inductive limits are exact, there is a canonical isomorphism of functors \[H_i({\bf a}^\infty,\bullet)\cong\ilim_{u\in\N} (H_i({\bf a}^u,\bullet)).\tag{2}\] If $u\in\N$, then the components of $K_\bullet({\bf a}^u)$ and $K^\bullet({\bf a}^u)$ are free of finite type, hence there are canonical isomorphisms $K_\bullet({\bf a}^u)\cong\hm{R}{K^\bullet({\bf a}^u)}{R}$ of complexes and\linebreak $\hm{R}{K^\bullet({\bf a}^u)}{R}\otimes_R\sq\cong\hm{R}{K^\bullet({\bf a}^u)}{\sq}$ of functors (\cite[II.2.7 Proposition 13; II.4.2 Proposition 2]{a}). Therefore, there is a canonical isomorphism of functors \[\ilim_{u\in\N}(H_i({\bf a}^u,\sq))\cong\ilim_{u\in\N}H_i(\hm{R}{K^\bullet({\bf a}^u)}{\sq})\tag{3}.\] If $I$ is an injective $R$-module, then $\hm{R}{\bullet}{I}$ is exact and thus commutes with formation of homology, so that there is a canonical isomorphism of $R$-modules \[\ilim_{u\in\N} H_i(\hm{R}{K^\bullet({\bf a}^u)}{I})\cong\ilim_{u\in\N} \hm{R}{H^i({\bf a}^u,R)}{I}.\tag{4}\] If $u\in\N$, then $H^i({\bf a}^u,R)$ is an $\ia$-torsion module (\cite[X.9.1 Proposition 1 Corollaire 2]{a}). Hence, there is a canonical isomorphism of functors \[\ilim_{u\in\N}\hm{R}{H^i({\bf a}^u,R)}{\bullet}\cong\ilim_{u\in\N}\hm{R}{H^i({\bf a}^u,R)}{\Gamma_{\ia}(\bullet)}.\tag{5}\]

Now, if $I$ is an injective $R$-module, then so is $\Gamma_{\ia}(I)$ by our hypothesis, and thus assembling the above, we get canonical isomorphisms of $R$-modules \[\check{H}^i({\bf a},I)\overset{(1)-(4)}\cong\ilim_{u\in\N}\hm{R}{H^{n-i}({\bf a}^u)}{I}\overset{(5)}\cong\]\[\ilim_{u\in\N}\hm{R}{H^{n-i}({\bf a}^u)}{\Gamma_{\ia}(I)}\overset{(1)-(4)}\cong\check{H}^i({\bf a},\Gamma_{\ia}(I)).\] Since the components of nonzero degree of the \v{C}ech cocomplex with respect to ${\bf a}$ of an $\ia$-torsion module are zero, it follows that $\check{H}^i({\bf a},\bullet)$ is effaceable for $i\in\N^*$, and thus $\ia$ is weakly proregular.
\end{proof}

The results from Section 1 together with \ref{wpr20} imply, for example, that ideals of finite type in absolutely flat rings or in $1$-dimensional local domains are weakly proregular. The statement about absolutely flat rings can be proven directly by first noting that idempotent elements are proregular and thus generate weakly proregular ideals (\cite[2.7]{schenzel}), and then using the fact that an ideal of finite type in an absolutely flat rings is generated by an idempotent (\cite[4.23]{lam1}). Interestingly, the same observation shows that the principal ideal $\ia$ in Example \ref{1.200} A) is weakly proregular. Finally, let us point out that \ref{wpr20} together with \ref{1.150} yields a new proof of the fact that ideals in noetherian rings are weakly proregular.

\smallskip

Now we turn to basic results on local cohomology for rings with ITI. The interplay between $\Gamma_{\ia}$ and local cohomology functors is described by the following result, proven analogously to \cite[2.1.7]{bs} on use of \ref{1.10}.

\begin{prop}\label{3.10}
Suppose $R$ has ITI with respect to $\ia$, let $i>0$ and let $M$ be an $R$-module. Then, $H^i_{\ia}(\Gamma_{\ia}(M))=0$, and the canonical morphism $H_{\ia}^i(M)\rightarrow H_{\ia}^i(M/\Gamma_{\ia}(M))$ is an isomorphism.
\end{prop}

In case $\ia$ is of finite type, this follows immediately from \ref{wpr20} and the definitions of weak proregularity and of \v{C}ech cohomology.\smallskip

Using the notion of triad sequence we get the important Comparison Sequence.

\begin{prop}\label{3.30}
Let $n\in\Z$, let $b\in R$ and $\ib\dfgl\langle b\rangle_R$, and suppose $R$ has ITI with respect to $\ia$ and with respect to $\ib$. Then, there is an exact sequence of functors $$0\longrightarrow H_{\ib}^1\circ H_{\ia}^{n-1}\longrightarrow H_{\ia+\ib}^n\longrightarrow\Gamma_{\ib}\circ H_{\ia}^n\longrightarrow 0.$$
\end{prop}

In \cite[3.5]{schenzel}, Schenzel gets the same conclusion under the hypothesis that $\ia$ is of finite type and that $\ia$, $\ib$ \textit{and} $\ia+\ib$ are weakly proregular. The Comparison Sequence allows for an inductive proof of an ITI variant of Hartshorne's Vanishing Theorem (\cite[3.3.3]{bs}), but this also follows immediately from \ref{wpr20}.

\smallskip

The next two results concern the behaviour of local cohomology under change of rings: the Base Ring Independence Theorem and the Flat Base Change Theorem. Both can be proven analogously to \cite[4.2.1; 4.3.2]{bs}, relying on a general variant of the vanishing result \cite[4.1.3]{bs}. The latter can be obtained in the noetherian case on use of the Mayer-Vietoris Sequence, and in general on use of the Comparison Sequence.

\begin{prop}\label{3.50}
Let $R\rightarrow S$ be a morphism of rings. Suppose $\ia$ is of finite type and $R$ has ITI with respect to ideals of finite type. Then, there is a canonical isomorphism of $\delta$-functors $$(H^i_{\ia S}(\bullet)\res_R)_{i\in\Z}\cong(H^i_{\ia}(\bullet\res_R))_{i\in\Z}.$$
\end{prop}

\begin{prop}\label{3.60}
Let $R\rightarrow S$ be a flat morphism of rings. Suppose $\ia$ is coherent and $S$ has ITI with respect to extensions of coherent ideals of $R$.\footnote{The hypothesis on $S$ is fulfilled if it has ITI with respect to ideals of finite type.} Then, there is a canonical isomorphism of $\delta$-functors $$(H^i_{\ia}(\bullet)\otimes_RS)_{i\in\Z}\cong(H^i_{\ia S}(\bullet\otimes_RS))_{i\in\Z}.$$
\end{prop}

Under the hypothesis that $\ia$ and $\ia S$ are weakly proregular, the conclusion of the Base Ring Independence Theorem is shown to hold in \cite[6.5]{yekutieli}, while the conclusion of the Flat Base Change Theorem follows immediately from the definitions of weak proregularity and \v{C}ech cohomology.

\smallskip

Up to now, we saw that several basic results on local cohomology can be generalised by replacing noetherianness with ITI properties (and maybe some coherence hypotheses). But alas!, this does not work for Grothendieck's Vanishing Theorem (\cite[6.1.2]{bs}). We end this article with a positive and a negative result in this direction for absolutely flat rings. Keep in mind that absolutely flat rings have ITI with respect to every ideal by \ref{1.120} (hence their ideals of finite type are weakly proregular by \ref{wpr20}) and that they are moreover coherent.

\begin{prop}\label{3.80}
Let $R$ be an absolutely flat ring.

a) If $\ia\subseteq R$ is an ideal of finite type, then $H^i_{\ia}=0$ for every $i>0$.

b) If $R$ is non-noetherian, then there exist an ideal $\ia\subseteq R$ not of finite type, an $R$-module $M$ and $i>\dim(M)$ such that $H^i_{\ia}(M)\neq 0$.
\end{prop}

\begin{proof}
An ideal $\ia\subseteq R$ is idempotent, hence $\Gamma_{\ia}(\bullet)\cong\hm{R}{R/\ia}{\bullet}$, and thus $H^i_{\ia}=0$ if and only if the $R$-module $R/\ia$ is projective. In case $\ia$ is of finite type, the $R$-module $R/\ia$ is flat and of finite presentation, hence projective. This shows a).

Assume now $R$ is non-noetherian and $H^i_{\ia}=0$ for every $i>0$. By the above, this implies projectivity of every monogeneous $R$-module, thus by \cite[VIII.8.2 Proposition 4]{a} the contradiction that $R$ is semisimple (and in particular noetherian). This shows b).
\end{proof}


\medskip
\noindent{\bf Acknowledgement:} For their various help during the writing of this article we thank Markus Brodmann, Paul-Jean Cahen, Neil Epstein, Shiro Goto, John Greenlees, Thomas Preu, Rodney Sharp, and Kazuma Shimomoto. We are grateful to the referee for his careful reading.


\end{document}